\documentclass[11pt,letterpaper,reqno]{amsart}
\usepackage{fullpage}
\usepackage[top=2cm, bottom=4.5cm, left=2.5cm, right=2.5cm]{geometry}
\usepackage{amsmath,amsthm,amsfonts,amssymb,amscd}
\usepackage{geometry,tikz,tikz-cd}
\usepackage{empheq}
\usepackage{lipsum}
\usepackage{lastpage}
\usepackage{thmtools}
\usepackage{enumerate}
\usepackage[shortlabels]{enumitem}
\usepackage{fancyhdr}
\usepackage{mathrsfs}
\usepackage{xcolor}
\usepackage{graphicx}
\usepackage{listings}
\usepackage{bbm}
\usepackage{array}
\usepackage{hyperref}
\usepackage[makeroom]{cancel}

\usepackage[utf8]{inputenc}
\usepackage{todonotes}
\usepackage{comment}

\usepackage[all]{xy}
\xyoption{matrix}
\xyoption{arrow}

\usetikzlibrary{arrows,decorations.pathmorphing,backgrounds,positioning,fit,petri}

\tikzset{help lines/.style={step=#1cm,very thin, color=gray},
help lines/.default=.5} 
\tikzset{thick grid/.style={step=#1cm,thick, color=gray},
thick grid/.default=1} 

\mathtoolsset{showonlyrefs=true}
\numberwithin{figure}{section}
\numberwithin{table}{section}

\theoremstyle{definition}

\theoremstyle{plain}
\newcommand{\thistheoremname}{}
\newtheorem*{genericthm*}{\thistheoremname}
\newenvironment{namedthm*}[1]
  {\renewcommand{\thistheoremname}{#1}%
   \begin{genericthm*}}
  {\end{genericthm*}}

\hypersetup{%
  colorlinks=true,
  linkcolor=blue,
  citecolor=blue,
  linkbordercolor={0 0 1}
}

\makeatletter
\NewDocumentCommand{\sump}{e{_}}
 {%
  \DOTSB
  \mathop{\IfNoValueTF{#1}{\sump@{}}{\sump@{#1}}}%
  \nolimits
 }
\newcommand{\sump@}[1]{\mathpalette\sump@@{#1}}
\newcommand{\sump@@}[2]{%
  \ifx#1\displaystyle
    {\sump@display{#2}}%
  \else
    \sum@\nolimits'_{#2}%
  \fi
}
\newcommand{\sump@display}[1]{%
  \sbox\z@{$\m@th\displaystyle\sum@\nolimits'$}%
  \sbox\tw@{$\m@th\displaystyle\sum@\limits_{#1}$}%
  \sbox\@tempboxa{$\m@th\displaystyle'$}
  \mathop{\sum@\nolimits' \kern-\wd\@tempboxa}\limits_{#1}%
  \ifdim\wd\z@>\wd\tw@
    \kern\dimexpr\wd\z@-\wd\tw@\relax
  \fi
}
\makeatother

\newcommand{\ZZ}{\mathbb{Z}}

\newcommand{\lrabs}[1]{\left\lvert #1 \right\lvert}
\newcommand{\lrp}[1]{\left(#1\right)}
\newcommand{\lrb}[1]{\left[#1\right]}
\newcommand{\lrfloor}[1]{\left\lfloor #1 \right\rfloor}

\allowdisplaybreaks

 \newtheorem{theorem}{Theorem}[section]
 
 \newtheorem{lemma}[theorem]{Lemma}
 \newtheorem{proposition}[theorem]{Proposition}
 \theoremstyle{definition}
 
 \theoremstyle{definition}
 
 \theoremstyle{remark}
\newtheorem{conjecture}[theorem]{\bf Conjecture}

 \numberwithin{equation}{section}

\usepackage{latexsym}
\usepackage{amssymb}

\newcommand{\ben}{\begin{equation}}
\newcommand{\een}{\end{equation}}

\allowdisplaybreaks

\DeclareMathOperator{\SL}{SL}

\DeclareMathOperator{\Egv}{Egv}

\DeclareMathOperator{\Tr}{Tr}

\setlist[enumerate]{leftmargin=*,widest=0}
\setlist[itemize]{leftmargin=*,widest=0}
\makeatletter
\def\subsection{\@startsection{subsection}{2}%
  \z@{.5\linespacing\@plus.7\linespacing}{.3\linespacing}%
  {\normalfont\bfseries}}

\def\subsubsection{\@startsection{subsubsection}{3}%
  \z@{.5\linespacing\@plus.7\linespacing}{.3\linespacing}%
  {\normalfont\bfseries}}
\makeatother

\begin{document}
\title{Non-repetition of second coefficients of Hecke polynomials}


\author{Archer Clayton}
\address{Department of Mathematics, Brigham Young University, Provo, UT 84602}
\email{ac727@byu.edu}

\author{Helen Dai}
\address{Department of Mathematics, Harvard University, 1 Oxford St, Cambridge, MA 02138}
\email{hdai@college.harvard.edu}

\author{Tianyu Ni}
\address{School of Mathematical and Statistical Sciences\\
Clemson University\\
Clemson, SC 29634-0975}
\email{tianyuni1994math@gmail.com}

\author{Erick Ross}
\address{School of Mathematical and Statistical Sciences\\
Clemson University\\
Clemson, SC 29634-0975}
\email{erickr@clemson.edu}

\author{Hui Xue}
\address{School of Mathematical and Statistical Sciences\\
Clemson University\\
Clemson, SC 29634-0975}
\email{huixue@clemson.edu}

\author{Jake Zummo}
\address{Department of Mathematics, The University of Chicago, Chicago, IL 60637}
\email{jzummo25@uchicago.edu}

\begin{abstract} 
    Let $T_m(N,2k)$ denote the $m$-th Hecke operator on the space $S_{2k}(\Gamma_0(N))$ of cuspidal modular forms of weight $2k$ and level $N$. 
    In this paper, we study the non-repetition of the second coefficient of the characteristic polynomial of  $T_m(N,2k)$. 
    We obtain results in the horizontal aspect (where $m$ varies), the vertical aspect (where $k$ varies), and the level aspect (where $N$ varies).
    Finally, we use these non-repetition results to extend a result of Vilardi and Xue on distinguishing Hecke eigenforms. 
\end{abstract}

\maketitle

\section{Introduction}

Let $S_{2k}(\Gamma_0(N))$ \cite[Section 3.1]{serrehecke} denote the space of cusp forms of weight $2k$ and level $N$, with $s(N, 2k)=\dim S_{2k}(\Gamma_0(N))$. 
For $q=e^{2\pi iz},$ let $f(z)=\sum_{m=1}^{\infty}c_f(m) q^m$ be the Fourier expansion of $f\in S_{2k}(\Gamma_0(N))$. For $m \geq 1$, the $m$-th Hecke operator $T_m(N, 2k)$ \cite[Proposition 10.2.5]{cohen-stromberg} acts on $f$ by 
\begin{align}
T_m(N, 2k) f(z)=\sum_{n=1}^{\infty}\left(\sum_{\substack{d\mid (m,n)\\ (d,N)=1}}d^{2k-1} c_f(mn/d^2)\right)q^n.
\end{align}

We are interested in studying various Hecke polynomials: the characteristic polynomials of the Hecke operators $T_m(N, 2k)$. We can express the Hecke polynomial associated to $T_m(N, 2k)$ as
$$T_m(N,2k)(x)=\sum_{n=0}^{s(N,2k)} (-1)^n a_n(T_m(N, 2k)) x^{s(N,2k)-n}.$$
Here, we will refer to $a_n(T_m(N, 2k))$ as the $n$-th coefficient of $T_m(N, 2k)(x)$. Note that for any given $n$, $a_n(T_m(N, 2k))$ is only defined when $s(N, 2k) \geq n$. However, this is not a restrictive condition; $s(N, 2k) < n$ for only finitely many pairs $(N,k)$ \cite[Theorem 1.1]{ross}. Now, observe that using this notation, $a_1(T_m(N, 2k))$ is just $\Tr T_m(N, 2k)$, which is a subject of great interest in the theory of modular forms. For example, the ``generalized Lehmer conjecture" from \cite[Conjecture 1.5]{rouse} predicts that  $\Tr T_m(N, 2k)$ never vanishes for $m$ coprime to $N$ and $2k \geq 12$, $2k \neq 14$. Here, the classical Lehmer conjecture comes from taking $N=1$, $2k=12$.
Another natural question to ask is whether the coefficients $a_n$ repeat values, and if so, which values are repeated. Recently, it was shown in \cite{chirjorznorepeat} that $a_1(T_2(1, 2k)) = \Tr T_2(1, 2k)$ takes no repeated values whenever $s(1, 2k) \geq 1$. Part of the motivation of \cite{chirjorznorepeat} comes from the work of \cite{Vilardi2018} and \cite{xue-zhu}, which showed that Hecke eigenforms of level one can be distinguished by their second or third Fourier coefficients, provided that Maeda's conjecture holds.  An important step in \cite{Vilardi2018} and \cite{xue-zhu} is to show that  $\Tr T_2(1,2k) \ne \Tr T_2(1,2\ell)$ and $\Tr T_3(1,2k)\ne \Tr T_3(1,2\ell)$ for $k\ne \ell$ and $s(1,2k)=s(1,2\ell)\ge1$. 

In this paper, we investigate corresponding questions on the second coefficient $a_2$ of Hecke polynomials. Since $T_m(N,2k)$ takes three parameters, we consider three aspects of these questions, by fixing two of the three parameters. The first question concerns the vertical non-repetition of values of $a_2$ (non-repetition in the weight aspect). We show that for any fixed odd level $N$, $a_2(T_2(N, 2k))$ and $a_2(T_4(N, 2k))$ do not repeat for different values of $k$. 
The second question concerns the horizontal non-repetition of $a_2$, where we restrict to $N = 1$ and fix $k$, while considering varying $m$. The third question concerns the level non-repetition of $a_2$: 
we consider
the case when $m=2$ and $k$ are fixed and consider varying prime levels. 
Lastly, we apply these non-repetition results to the problem of distinguishing Hecke eigenforms. In the main result of \cite{Vilardi2018}, Vilardi and Xue showed that, assuming Maeda's conjecture, Hecke eigenforms of level one are determined by their $2$nd Fourier coefficient. We use the non-repetition results in this paper to extend their result to the $4$-th Fourier coefficient, and discuss how the same methods can apply to the $m$-th Fourier coefficient for any given $m \ge 2$. 

We now state the results precisely. In Section \ref{sect:T2nonrepetition}, we show that $a_2(T_2(N, 2k))$ takes no repeated values in the vertical aspect.
\begin{theorem}\label{theorem:main}
    Let $N \ge 1$ odd be fixed, and consider $k \geq 1$ such that $s(N,2k) \ge 2$. Then 
    $a_2(T_2(N, 2k))$
    takes no repeated values as $k$ varies.
    In fact, $a_2(T_2(N, 2k))$ is a strictly decreasing function of $k$.
\end{theorem}

Next, in Section \ref{sect:T4nonrepetition}, we show that $a_2(T_4(N,2k))$ takes no repeated values either. The second coefficient $a_2(T_m(N,2k))$ has slightly different behavior depending on whether $m$ is a perfect square or not. So we also include this case of $m=4$ to highlight this distinction.
\begin{theorem}\label{theorem:main4}
    Let $N \ge 1$ odd be fixed, and consider $k \geq 1$ such that $s(N,2k) \ge 2$. Then 
    $a_2(T_4(N, 2k))$
    takes no repeated values as $k$ varies.
\end{theorem}

In Section \ref{sect:T2T3}, we fix the weight $2k$  and level $N = 1$ and prove the following horizontal result.
\begin{theorem}\label{theorem:T2T3notequal}
    For any $k\geq 12$ and $k\neq 13$,
    \begin{equation}
        a_2(T_3(1,2k))< a_2(T_2(1,2k)).
    \end{equation}
\end{theorem}
Observe that $a_2(T_m(1, 2k))$ is only defined for $s(1,2k)\geq 2$, hence the assumptions on $k$ in this result. 

Next, in Section \ref{sect:primelevel}, we 
fix $k$ and $m=2$ and consider varying prime levels (including one).
\begin{theorem} \label{theorem:primelevel}
 If $k\ge 58$, then
    \begin{equation}
      a_{2}(T_2(p,2k))> a_{2}(T_2(q,2k))
    \end{equation}
 for all odd primes $p< q$, where $p$ can be taken to be $1$.
\end{theorem}

Finally, in Section \ref{sect:distinguish}, we use Theorems \ref{theorem:main} and \ref{theorem:main4} to extend the main result of \cite{Vilardi2018} to $m=2,4$.
We also discuss how one can use our method to extend this result to any given $m \ge 2$.
\begin{theorem} \label{theorem:distinguish-eigenforms}
    Let $m = 2 \text{ or }4$ be fixed, and let $f=\sum_{n\ge1} c_f(n) q^n\in S_{2k_1}(\SL_2(\ZZ))$ and $g=\sum_{n\ge1} c_g(n) q^n\in S_{2k_2}(\SL_2(\ZZ))$ be normalized Hecke eigenforms. Assume that the characteristic polynomials of the $T_m(1,2k_i)$ are irreducible. Then
    \begin{equation}
    f=g \quad \text{iff} \quad c_f(m) = c_g(m).
    \end{equation}
\end{theorem}


\section{Preliminaries}

In this section, we state the necessary results used to show the non-repetition of the $a_2$ coefficients of Hecke polynomials in various scenarios. Much of this background information can be found in \cite{previousPaper}.  Our work hinges on the following expression for $a_2(T_m(N, 2k))$ in terms of traces of Hecke operators:

\begin{proposition}[{\cite[Proposition 2.1]{previousPaper}}]  \label{prop:a2-formula}
Suppose $\gcd(N,m)=1$. Then    
\begin{align}\label{eq:step1}
        a_2(T_m(N,2k)) 
        &= \frac{1}{2}\left[(\Tr T_m(N,2k))^2- \sum_{d\mid m} d^{2k-1} \Tr T_{m^2/d^2}(N,2k) \right].
   \end{align}
\end{proposition}
\begin{proof}
  Let $\alpha_1, \dots, \alpha_{s(N,2k)}$ be the eigenvalues of the Hecke operator $T_m(N, 2k)$.   Then
   \begin{align}
        a_2(T_m(N,2k)) &= \sum_{1\leq i<j\leq s(N,2k)}\alpha_i\alpha_j \\
        &= \frac{1}{2}\left[ \left( \sum_i \alpha_i\right)^2-\sum_{j}\alpha_j^2 \right] \\
        &=\frac{1}{2}\left[(\Tr T_m(N,2k))^2-\Tr (T_m(N,2k)^2)\right].
   \end{align}
   Then using the well-known Hecke operator composition formula \cite[Theorem~10.2.9]{cohen-stromberg}, we have
    \begin{align}
        a_2(T_m(N,2k)) &= \frac{1}{2}\left[(\Tr T_m(N,2k))^2-\Tr (T_m(N,2k)^2)\right] \\
        &= \frac{1}{2}\left[(\Tr T_m(N,2k))^2- \sum_{d\mid m} d^{2k-1} \Tr T_{m^2/d^2}(N,2k) \right],
   \end{align}
   as desired.
\end{proof}


Next we introduce the Eichler-Selberg trace formula (\cite[p.~370-371]{knightly2006traces} and \cite[(34)]{serrehecke}), which we will use to compute $\Tr T_m(N,2k)$ explicitly. 

\begin{proposition}\label{eq:eich-selb}
Suppose $\gcd(N,m)=1$. Then
$$\Tr T_m(N, 2k) = A_{1,m}(N, 2k) + A_{2,m}(N, 2k) + A_{3,m}(N, 2k) + A_{4,m}(N, 2k),$$
where
\begin{align}
    A_{1,m}(N, 2k)&=\begin{cases}
        \frac{2k-1}{12}\psi(N)m^{k-1} & {\rm if } ~m~ {\rm is~a~perfect~square},\\
        0 & {\rm otherwise},\\
    \end{cases}\\
    A_{2,m}(N, 2k)&=-\frac{1}{2}\sum_{t^2<4m}P_{2k}(t,m)\sum_n h_w\left(\frac{t^2-4m}{n^2}\right)\mu(t,n,m),\\
    A_{3,m}(N, 2k)&=-\frac{1}{2}\sum_{d \mid m}\min(d,m/d)^{2k-1}\sump_{\tau}\varphi(\gcd(\tau,N/\tau))\mathbf{1}_N(y),\\
    A_{4,m}(N, 2k)&=\begin{cases}
        \sum_{\substack{c|m\\
    \gcd(N,m/c)=1}}c& {\rm if }~k=1,\\
    0 & {\rm otherwise}.
    \end{cases}
\end{align}
Here, we have
\begin{itemize}
    \item $\psi(N)= [\SL_2(\mathbb{Z}):\Gamma_0(N)]= N\prod_{p \mid N} \left(1 + \frac{1}{p}\right)$ \cite[p.~21]{diamond2005first},
    \item $n$ in the summation within $A_{2,m}(N, 2k)$ runs through all positive integers such that $n^2 \mid (t^2-4m)$ and $\frac{t^2-4m}{n^2}\equiv0,1\pmod 4$,
    \item $P_{2k}(t,m)$ is the coefficient of $x^{2k-2}$ in the power series expansion of $(1 - tx + mx^2)^{-1}$,
    \item $h_w\left(\frac{t^2-4m}{n^2}\right)$ the weighted class number of discriminant $\frac{t^2 - 4m}{n^2}$,
    \item  $\mu(t,n,m)=\frac{\psi(N)}{\psi(N/N_{n})}\sum_{c\mod N}1$, where $N_n=\gcd(N,n)$, and $c$ runs through all elements of $(\mathbb{Z}/N\mathbb{Z})^{\times}$ which lift to solutions of $c^2-tc+m\equiv0\pmod {NN_n}$,
    \item $\varphi$ is the Euler totient function,
    \item $ \sump_{\tau} $ means that $\tau$ runs through all positive divisors such that 
    $\gcd(\tau,N/\tau)$ divides $(d-\frac{m}{d})$,
    \item $y$ is the unique integer mod $\frac{N}{\gcd(\tau,N/\tau)}\mathbb{Z}$ such that $y\equiv d\mod \tau \mathbb{Z}$ and $y\equiv \frac{m}{d}\mod \frac{N}{\tau} \mathbb{Z}$,
    \item $\mathbf{1}_N(n)=1$ if $\gcd(n,N)=1$ and $\mathbf{1}_N(n)=0$ if $\gcd(n,N)>1$.
\end{itemize}
\end{proposition}
For a more in-depth discussion of this formula and its proof, see \cite{knightly2006traces}. 

Finally, the following proposition provides a closed-form expression for the dimension $s(N,2k)$. Note that since $T_1(N,2k)$ is the identity operator, $s(N,2k) = \Tr T_1(N,2k)$ is one of the terms appearing in Proposition \ref{prop:a2-formula}. 

\begin{proposition}[{\cite[Corollary 7.4.3]{cohen-stromberg}}] \label{prop:dim-formula}
    Let $N \ge 1$ and $k \ge 1$. Then
    \begin{align}
        s(N,2k) \,=\, & \frac{2k-1}{12} \psi(N) - \frac12  \sum_{d\mid N} \varphi(\gcd(d,N/d)) + \delta_{k,1} \\
        & - \varepsilon(N/9)\, c_3(k) \prod_{p \mid N} \lrp{1 + \lrp{\frac{-3}{p}} } 
        - \varepsilon(N/4)\, c_4(k) \prod_{p \mid N} \lrp{1 + \lrp{\frac{-4}{p}} }.
    \end{align}
    Here, $\delta_{k,1}$ denotes the Kronecker delta; $\varepsilon(x) = 1$ if $x \not\in \ZZ$, $\varepsilon(x) = 0$ if $x \in \ZZ$; $\lrp{\frac{\cdot}{p}}$ denotes the Legendre symbol; $c_3(k) = \frac{2k-1}{3} - \lrfloor{\frac{2k}{3}} $; and $c_4(k) = \frac{2k-1}{4} - \lrfloor{\frac{2k}{4}} $.
    
\end{proposition}

\section{Non-repetition of \texorpdfstring{$a_2(T_2(N,2k))$}{a2(T2(N,2k))} for varying weight}\label{sect:T2nonrepetition}

In this section, we prove Theorem \ref{theorem:main}, which claims that for any fixed odd $N \geq 1$, $a_2(T_2(N, 2k))$ is a strictly decreasing function of $k$. Here we are only considering $k \ge 1$ such that $s(N,2k) \ge 2$ (which will be all but finitely many values of $k$). 
We prove this strictly decreasing property by showing that $a_2(T_2(N, 2k+2)) - a_2(T_2(N, 2k)) < 0$ for sufficiently large $k$. We will then check the finitely many remaining cases by computer.

Specializing Proposition \ref{prop:a2-formula} to $m=2$, we obtain
\begin{align}
    &a_2(T_2(N, 2k + 2)) - a_2(T_2(N, 2k))\\ 
    = &\frac{1}{2} 
    \lrb{(\Tr T_2(N, 2k+2))^2-\Tr T_4(N, 2k+2) - 2^{2k+1} s(N, 2k+2) } \\
    & -\frac{1}{2}
    \lrb{(\Tr T_2(N, 2k))^2-\Tr T_4(N, 2k) - 2^{2k-1} s(N,2k)}  \\
    = &\lrp{ \frac12
        (\Tr T_2(N, 2k+2))^2 - \frac12 (\Tr T_2(N, 2k))^2 } 
    + \lrp{ - \frac12 \Tr T_4(N, 2k+2) + \frac12 \Tr T_4(N, 2k)} \\
    &+ \lrp{-  4^{k} s(N, 2k+2) + 4^{k-1}          s(N, 2k)
    }.
    \label{eqn:a2_diff}
\end{align}
We will then bound the following three components of \eqref{eqn:a2_diff} separately:
\begin{itemize}
    \item $\frac{1}{2}(\Tr T_2(N, 2k+2))^2 - \frac{1}{2}(\Tr T_2(N, 2k))^2$,
    \item $-\frac{1}{2}\Tr T_4(N, 2k+2) + \frac{1}{2}\Tr T_4(N, 2k)$,
    \item $-4^{k}s(N, 2k+2) + 4^{k-1}s(N, 2k)$.
\end{itemize}
We will express these bounds in terms of the following $\theta_i(N)$ error functions. Here, $\omega(N)$ denotes the number of distinct prime divisors of $N$ and $\sigma_0(N)$ denotes the number of divisors of $N$. Observe that each of these $\theta_i(N)$ tends to $0$ as $N \rightarrow \infty$ since $\psi(N) \geq N$ and 
$2^{\omega(N)}, \, \sigma_0(N) = O(N^\varepsilon)$ for all $\varepsilon > 0$. 
\begin{equation} \label{eqn:theta-i-def}
    \begin{alignedat}{5}
        \theta_1(N) &:= \frac{ \sqrt{N} 2^{\omega(N)} }{\psi(N)},
        \qquad\qquad
        & \theta_2(N) &:= \frac{\sigma_0(N)}{\psi(N)}, 
        \qquad\qquad 
        & \theta_3(N) &:= \frac{4^{\omega(N)}}{\psi(N)}, 
        \\
        \theta_4(N) &:= \frac{2^{\omega(N)}}{\psi(N)}, 
        \qquad\qquad 
        & \theta_5(N) &:= \frac{1}{\psi(N)}.
        & &
    \end{alignedat}
\end{equation}

For the first component of \eqref{eqn:a2_diff}, we have by \cite[Proposition 4.6]{previousPaper} that
\begin{align}
    & \frac{1}{2}(\Tr T_2(N, 2k+2))^2 - \frac{1}{2}(\Tr T_2(N, 2k))^2 \\
    \leq &\frac{1}{2}(\Tr T_2(N, 2k+2))^2 \\
    \leq &
    \frac12 \lrb{  
    3 + (2^{k+5/2} + 1) 2^{\omega(N)}
    }^2 \\
    \leq & 
    \frac12 \lrb{ \lrp{ 
    3 + (8\sqrt{2} + 1)  2^{\omega(N)} } 2^{k-1}
    }^2 \\
    =&  \frac12\cdot 4^{k-1} \lrb{
        9 + 6(8\sqrt{2} + 1)  2^{\omega(N)} +  (8\sqrt{2} + 1)^2 4^{\omega(N)}
    } \\
    =& 4^{k-1} \psi(N) \lrb{ 
        \frac92 \theta_5(N) + \lrp{24\sqrt{2} + 3} \theta_4(N) + \lrp{\frac{129}{2} + 8\sqrt{2}} \theta_3(N)
    }.
    \label{eqn:component1}
\end{align}

For the second component of \eqref{eqn:a2_diff}, we have by the Eichler-Selberg trace formula,
\begin{align}
& -\frac{1}{2}\Tr T_4(N, 2k+2) + \frac{1}{2}\Tr T_4(N, 2k) \\
=& -\frac{1}{2}\left(\frac{2k+1}{12}\psi(N)4^{k}+A_{2,4}(N,2k+2)+A_{3,4}(N,2k+2)+A_{4,4}(N,2k+2) 
 \right) \\
&+\frac{1}{2}\left(\frac{2k-1}{12}\psi(N)4^{k-1}+A_{2,4}(N,2k)+A_{3,4}(N,2k)+A_{4,4}(N,2k)\right) \\
\leq & \lrp{ \frac{-6k-5}{24} } \psi(N) 4^{k-1}  \\
&+ \frac12 \Big(
    \lrabs{A_{2,4}(N,2k+2)} + \lrabs{A_{2,4}(N,2k)} - A_{3,4}(N,2k+2) + A_{4,4}(N,2k)
\Big). \label{eqn:temp-T4-error}
\end{align}
Here, we used the facts that $A_{3,4}(N,2k) \leq 0$ and $-A_{4,4}(N,2k+2) \leq 0$.
We also already have bounds on each of the $A_i$ terms in \eqref{eqn:temp-T4-error}. From the proof of \cite[Proposition 4.7]{previousPaper},
\begin{equation}
\begin{alignedat}{1}
    |A_{2,4}(N,2k)| &\leq 4^{k-1}  \frac{61}{3} 2^{\omega(N)}, \\
    |A_{2,4}(N,2k+2)| &\leq 4^{k}  \frac{61}{3} 2^{\omega(N)} = 4^{k-1} \frac{244}{3} 2^{\omega(N)},
    \\
    -A_{3,4}(N, 2k+2) &\leq 2\sigma_0(N)+4^{k}2^{\omega(N)}\sqrt{N} \leq 4^{k-1} \lrp{ 2 \sigma_0(N) + 4 \cdot  2^{\omega(N)}\sqrt{N}},  \\
    A_{4,4} &\leq 7 \leq 4^{k-1} \cdot 7.
\end{alignedat}
\end{equation}
Then substituting these bounds into \eqref{eqn:temp-T4-error}, we obtain
\begin{align}
& -\frac{1}{2}\Tr T_4(N, 2k+2) + \frac{1}{2}\Tr T_4(N, 2k) \\
&\leq \psi(N) 4^{k-1} \lrb{\frac{-6k-5}{24} + \frac{305}{6} \theta_4(N) +  \theta_2(N) + 2\, \theta_1(N) + \frac72 \theta_5(N)  }.   
 \label{eqn:component2}
\end{align}

Finally, we bound the third component of \eqref{eqn:a2_diff}.
Observe that from the dimension formula given in Proposition \ref{prop:dim-formula}, we immediately have 
\begin{equation} \label{eqn:dim-formula-terms-bounds}
\begin{alignedat}{3}
    \lrabs{c_3(k)} &\leq \frac13,  \qquad\qquad &
    \prod_{p\mid N} \lrp{1 + \lrp{\frac{-3}{p}}} &\leq 2^{\omega(N)} ,  \\
    \lrabs{c_4(k)} &\leq \frac14, \qquad\qquad &
    \prod_{p\mid N} \lrp{1 + \lrp{\frac{-4}{p}}} &\leq 2^{\omega(N)}.
\end{alignedat}
\end{equation}
Additionally, using multiplicativity in $N$, it is straightforward to show (e.g. in 
\cite[Definition~12\textquotesingle(B),~Equation~(7)]{martin}) that 
\begin{align}
     \sum_{d\mid N} \varphi(\gcd(d,N/d)) \leq 2^{\omega(N)} \sqrt{N}.
\end{align}
Applying each of these bounds to the dimension formula given in Proposition \ref{prop:dim-formula} yields
\begin{align}
    &-4^{k}s(N, 2k+2) + 4^{k-1}s(N, 2k) \\
    =& -4^k \Bigg[ 
    \frac{2k+1}{12} \psi(N) - \frac12  \sum_{d\mid N} \varphi(\gcd(d,N/d)) + \delta_{k+1,1} \\
    & - \varepsilon(N/9)\, c_3(k+1) \prod_{p \mid N} \lrp{1 + \lrp{\frac{-3}{p}} } 
        - \varepsilon(N/4)\, c_4(k+1) \prod_{p \mid N} \lrp{1 + \lrp{\frac{-4}{p}} }
    \Bigg] \\
    & + 4^{k-1} \Bigg[ 
    \frac{2k-1}{12} \psi(N) - \frac12  \sum_{d\mid N} \varphi(\gcd(d,N/d)) + \delta_{k,1} \\
    & - \varepsilon(N/9)\, c_3(k) \prod_{p \mid N} \lrp{1 + \lrp{\frac{-3}{p}} } 
    - \varepsilon(N/4)\, c_4(k) \prod_{p \mid N} \lrp{1 + \lrp{\frac{-4}{p}} }
    \Bigg] \\
    \leq & \, 4^{k-1} \lrb{ -\frac{8k+4}{12} \psi(N)  + \frac{2k-1}{12} \psi(N) + \frac32  2^{\omega(N)} \sqrt{N} + 1 + \lrp{\frac43 + \frac44 + \frac13 + \frac14 } 2^{\omega(N)} } \\
    = & \, 4^{k-1} \psi(N) \lrb{ \frac{-6k-5}{12}  + \frac32 \theta_1(N) + \theta_5(N) + \frac{35}{12} \theta_4(N) }. \label{eqn:component3}
\end{align}

With these bounds on the three components of \eqref{eqn:a2_diff}, we are now able to prove Theorem \ref{theorem:main}.

{
\renewcommand{\thetheorem}{\ref{theorem:main}}
\begin{theorem}
    Let $N \ge 1$ odd be fixed, and consider $k \geq 1$ such that $s(N,2k) \ge 2$. Then 
    $a_2(T_2(N, 2k))$
    takes no repeated values as $k$ varies.
    In fact, $a_2(T_2(N, 2k))$ is a strictly decreasing function of $k$.
\end{theorem}
\addtocounter{theorem}{-1}
}

\begin{proof}

Applying the bounds \eqref{eqn:component1}, \eqref{eqn:component2}, and \eqref{eqn:component3} to \eqref{eqn:a2_diff} yields the formula
\begin{align}
    & a_2(T_2(N, 2k + 2)) - a_2(T_2(N, 2k)) \\
    \leq & \, 4^{k-1} \psi(N) \Bigg[
        \lrp{ 
            \frac92 \theta_5(N) + \lrp{24\sqrt{2} + 3} \theta_4(N) + \lrp{\frac{129}{2} + 8\sqrt{2}} \theta_3(N)
        } \\
        &\qquad\qquad +
        \lrp{
            \frac{-6k-5}{24} + \frac{305}{6} \theta_4(N) +  \theta_2(N) + 2\, \theta_1(N) + \frac72 \theta_5(N)  
        } \\
        &\qquad\qquad +
        \lrp{ 
            \frac{-6k-5}{12}  + \frac32 \theta_1(N) + \theta_5(N) + \frac{35}{12} \theta_4(N) 
        }
    \Bigg] \\
    =& \, 4^{k-1} \psi(N) \lrb{ \frac{-6k-5}{8} + E(N) },  \label{eqn:a2_diff_EN}
\end{align}
where
\begin{align} \label{eqn:EN-def}
    E(N) = \frac72 \theta_1(N) + \theta_2(N) + \lrp{\frac{129}{2} + 8\sqrt{2}} \theta_3(N) + \lrp{24\sqrt{2} + \frac{323}{6} } \theta_4(N) + 9\, \theta_5(N).
\end{align}

We can also give an explicit numerical upper bound on $E(N)$. Using an identical argument as in \cite[Lemma 2.4]{ross}, one can show the following bounds. (These mirror \cite[Lemma 4.2]{previousPaper}, but with smaller constants; see \cite{ross-code}.)
\begin{align}
    2^{\omega(N)} \leq 4.862 \cdot N^{1/4}  \qquad \text{and} \qquad \sigma_0(N) \leq 8.447 \cdot N^{1/4}. \label{eqn:temp-mult-bounds}
\end{align} 
Additionally, recall that $\psi(N) \geq N$.
Then, using the definitions of the $\theta_i(N)$ \eqref{eqn:theta-i-def}, we substitute these bounds into \eqref{eqn:EN-def} to obtain
\begin{align} 
    E(N) \leq&\, 
    \frac72 \cdot 4.862 \cdot N^{-1/4} 
    + 8.447 \cdot N^{-3/4} 
    + \lrp{\frac{129}{2} + 8\sqrt{2}} \cdot 4.862^2 \cdot N^{-1/2} \\
    &
    + \lrp{24\sqrt{2} + \frac{323}{6} } \cdot 4.862 \cdot N^{-3/4} 
    + 9 \cdot N^{-1} \\
    <&\, \frac{11}{8} \qquad \text{ for } N \geq 3,\!392,\!663.
\end{align}
This means that since $\frac{-6k-5}{8} \leq \frac{-11}{8}$, we have by \eqref{eqn:a2_diff_EN} that $a_2(T_2(N, 2k + 2)) - a_2(T_2(N, 2k)) < 0$ for all $N \geq 3,\!392,\!663$ (independently of $k$). This verifies that for each odd $N \geq 3,\!392,\!663$, $a_2(T_2(N, 2k))$ is a strictly decreasing function of $k$.
Then for each of the remaining fixed odd values of $N < 3,\!392,\!663$, we will have $\frac{6k+5}{8} > E(N)$ (and hence $a_2(T_2(N, 2k + 2)) - a_2(T_2(N, 2k)) < 0$) for sufficiently large $k$, say $k \geq k_N$. This verifies that $a_2(T_2(N, 2k))$ is strictly decreasing for $k \geq k_N$.
We then verify by computer that $a_2(T_2(N, 2k))$ is strictly decreasing for  $1 \le k \le k_N$ such that $s(N,2k)\ge 2$; see \cite{ross-code} for the code. 

This shows that for any fixed odd $N$, $a_2(T_2(N, 2k))$ is a strictly decreasing function of $k$,  completing the proof.
\end{proof}

\section{Non-repetition of \texorpdfstring{$a_2(T_4(N,2k))$}{a2(T4(N,2k))} for varying weight}\label{sect:T4nonrepetition}

In this section, we prove Theorem \ref{theorem:main4}, showing the non-repetition of $a_2(T_4(N, 2k))$ as a function of $k$. We prove this result by showing that $a_2(T_4(N, 2k))$ is strictly increasing for sufficiently large $k$. We then check the finitely many remaining cases by computer.  
The details are very similar to that of the previous section, so we forgo repeating every detail of all the same bounding arguments. 

Specializing Proposition \ref{prop:a2-formula} to $m=4$, we obtain
\begin{align}
    &a_2(T_4(N, 2k + 2)) - a_2(T_4(N, 2k))\\ 
    =\, &\frac{1}{2} 
    \lrb{(\Tr T_4(N, 2k+2))^2-\Tr T_{16}(N, 2k+2) -  2^{2k+1} \Tr T_4(N, 2k+2) -  4^{2k+1} s(N, 2k+2) } \\
    & -\frac{1}{2} 
    \lrb{(\Tr T_4(N, 2k))^2-\Tr T_{16}(N, 2k) -  2^{2k-1} \Tr T_4(N, 2k) -  4^{2k-1} s(N, 2k) } \\
    =\, & \frac12 \Big( (\Tr T_4(N, 2k+2))^2 - (\Tr T_4(N, 2k))^2 \Big) + \frac12 \Big( - \Tr T_{16}(N, 2k+2) + \Tr T_{16}(N, 2k) \Big)  \\
    & + \frac12 \lrp{- 2^{2k+1} \Tr T_4(N, 2k+2) + 2^{2k-1} \Tr T_4(N, 2k) } 
    + \frac12 \lrp{- 4^{2k+1} s(N, 2k+2) + 4^{2k-1} s(N, 2k) }.  \label{eqn:a2_4_diff}
\end{align}
We will then bound the following four components of \eqref{eqn:a2_4_diff} separately:
\begin{itemize}
    \item $\frac{1}{2}(\Tr T_4(N, 2k+2))^2 - \frac{1}{2}(\Tr T_4(N, 2k))^2$,
    \item $-\frac{1}{2}\Tr T_{16}(N, 2k+2) + \frac{1}{2}\Tr T_{16}(N, 2k)$,
    \item $- \frac12 \cdot 2^{2k+1} \Tr T_4(N, 2k+2) + \frac12 \cdot 2^{2k-1} \Tr T_4(N, 2k)$,
    \item $-\frac12 \cdot 4^{2k+1}s(N, 2k+2) + \frac12 \cdot 4^{2k-1}s(N, 2k)$.
\end{itemize}

For the first component of \eqref{eqn:a2_4_diff}, we can again follow the argument of \cite[Proposition 4.7]{previousPaper} to obtain \eqref{eqn:temp-tr-t4-bounds} below. See also the proof of \cite[Proposition 5.5]{ross-xue} for the specific bounding details.
\begin{equation} \label{eqn:temp-tr-t4-bounds}
\Tr T_4(N, 2k) = \psi(N) 4^{k} \lrp{\frac{2k-1}{24} + E_{N,k}}, \quad \text{where} \quad \lrabs{E_{N,k}} \leq \frac{17}{2} \theta_4(N) + \frac14 \theta_1(N).
\end{equation}
This means that 
\begin{align}
    & \ \ \ \frac{1}{2}(\Tr T_4(N, 2k+2))^2 - \frac{1}{2}(\Tr T_4(N, 2k))^2 \\
    &= \frac{\psi(N)^2 16^{k+1}}{2} \lrb{ 
        \lrp{\frac{2k+1}{24}}^2 + \frac{2k+1}{12} E_{N,k+1} + E_{N,k+1}^2
    } \\
    & \ \ \ -
    \frac{\psi(N)^2 16^{k}}{2} \lrb{ 
        \lrp{\frac{2k-1}{24}}^2 + \frac{2k-1}{12} E_{N,k} + E_{N,k}^2
    } \\
    &\geq \frac{ \psi(N)^2 16^{k} }{2} \lrb{ 
        15 \lrp{\frac{2k+1}{24}}^2 - 16\frac{2k+1}{12} \lrabs{ E_{N,k+1} } - \frac{2k+1}{12} \lrabs{ E_{N,k} } - E_{N,k}^2
    } \\
    &\geq \frac{ \psi(N)^2 16^{k} (2k+1) }{2} \lrb{ 
        \frac{10k+5}{192} - \frac{16}{12} \lrabs{ E_{N,k+1} } - \frac{1}{12} \lrabs{ E_{N,k} } - \frac{1}{3} E_{N,k}^2
    } 
    \\
    &\geq \frac{ \psi(N)^2 16^{k} (2k+1) }{2} \lrb{ 
        \frac{10k+5}{192}  - \frac{17}{12} \lrp{\frac{17}{2} \theta_4(N) + \frac14 \theta_1(N)}  - \frac{1}{3} \lrp{\frac{17}{2} \theta_4(N) + \frac14 \theta_1(N)}^2
    }. \qquad \label{eqn:T4-comp1}
\end{align}

For the second component of \eqref{eqn:a2_4_diff}, we can similarly compute the following upper and lower bounds on $\Tr T_{16}(N,2k)$; see also \cite[Lemmas 6.1]{ross-xue} and \cite[Lemma 6.2]{ross-xue} for the specific bounding details:
\begin{align}
    \Tr T_{16}(N,2k) &\leq \psi(N)^2 16^k (2k-1) \lrb{\frac{4309}{192} \theta_5(N) }, \\
    \Tr T_{16}(N,2k) &\geq \psi(N)^2 16^{k} (2k-1) \lrb{\frac{-4211}{192} \theta_5(N)}.
\end{align}
These two $\Tr T_{16}$ bounds mean that
\begin{align}
    & \ \ \ -\frac{1}{2}\Tr T_{16}(N, 2k+2) + \frac{1}{2}\Tr T_{16}(N, 2k) \\
    &\geq -\frac{ \psi(N)^2 16^{k+1} (2k+1) }{2} \lrb{\frac{4309}{192} \theta_5(N) } - \frac{ \psi(N)^2 16^{k} (2k-1) }{2} \lrb{\frac{4211}{192} \theta_5(N)} \\
    &\geq \frac{ \psi(N)^2 16^{k} (2k+1) }{2} \lrb{- \frac{73155}{192} \theta_5(N)}. \label{eqn:T4-comp2}
\end{align}

For the third component of \eqref{eqn:a2_4_diff}, we have from \eqref{eqn:temp-tr-t4-bounds} that
\begin{align}
    & \ \ \ - \frac12 2^{2k+1} \Tr T_4(N, 2k+2) + \frac12 2^{2k-1} \Tr T_4(N, 2k)\\
    &= - \frac12 2^{2k+1} \psi(N) 4^{k+1} \lrb{\frac{2k+1}{24} + E_{N,k+1}} + \frac12 2^{2k-1} \psi(N) 4^k \lrb{\frac{2k-1}{24} + E_{N,k}} \\
    &\geq \frac{ \psi(N) 16^k }{2} \lrb{  -8 \frac{2k+1}{24} - 8 \lrabs{E_{N,k+1}} -  \frac12 \lrabs{E_{N,k}} } \\
    &\geq \frac{ \psi(N)^2 16^k (2k+1) }{2} \lrb{  -\frac{1}{3} \theta_5(N) - \frac83 \lrabs{E_{N,k+1}}  \theta_5(N) -  \frac16 \lrabs{E_{N,k}}  \theta_5(N) } \\
    &\geq \frac{ \psi(N)^2 16^k (2k+1) }{2} \lrb{  -\frac{201}{8} \theta_5(N) }. \label{eqn:T4-comp3}
\end{align}
In the last step, we used the fact from \eqref{eqn:temp-tr-t4-bounds} that $\lrabs{E_{N,k}}, \lrabs{E_{N,k+1}} \leq \frac{17}{2} \theta_4(N) + \frac14 \theta_1(N) \leq \frac{35}{4}$.

Finally, for the fourth component of \eqref{eqn:a2_4_diff}, we have by Proposition \ref{prop:dim-formula} that
\begin{align}
    &\ \ \ -\frac12\cdot 4^{2k+1}s(N, 2k+2) + \frac12 \cdot 4^{2k-1}s(N, 2k) \\
    &\geq -\frac12\cdot 4^{2k+1} s(N, 2k+2) \\
    &= -\frac12\cdot  4^{2k+1} \Bigg[ 
    \frac{2k+1}{12} \psi(N) - \frac12  \sum_{d\mid N} \varphi(\gcd(d,N/d)) + \delta_{k+1,1} \\
    & \ \ \ - \varepsilon(N/9)\, c_3(k+1) \prod_{p \mid N} \lrp{1 + \lrp{\frac{-3}{p}} } 
        - \varepsilon(N/4)\, c_4(k+1) \prod_{p \mid N} \lrp{1 + \lrp{\frac{-4}{p}} }
    \Bigg] \\
    &\geq -\frac12\cdot 4^{2k+1} \lrb{ \frac{2k+1}{12} \psi(N) + \frac13 \cdot 2^{\omega(N)} + \frac14 \cdot 2^{\omega(N)} } \\
    &\geq \frac{\psi(N) 16^k}{2} \lrb{ -4 \frac{2k+1}{12} - \frac73  } \\
    &\geq \frac{\psi(N)^2 16^k (2k+1)}{2} \lrb{ -\frac{10}{9} \theta_5(N) }. \label{eqn:T4-comp4}
\end{align}

We can then use these bounds on the four components of \eqref{eqn:a2_4_diff} in order to prove Theorem \ref{theorem:main4}.

{
\renewcommand{\thetheorem}{\ref{theorem:main4}}
\begin{theorem}
    Let $N \ge 1$ odd be fixed, and consider $k \geq 1$ such that $s(N,2k) \ge 2$. Then 
    $a_2(T_4(N, 2k))$
    takes no repeated values as $k$ varies.
\end{theorem}
\addtocounter{theorem}{-1}
}

\begin{proof}
Applying the bounds \eqref{eqn:T4-comp1}, \eqref{eqn:T4-comp2}, \eqref{eqn:T4-comp3}, and \eqref{eqn:T4-comp4} to \eqref{eqn:a2_4_diff}, we obtain
\begin{align}
    &\ \ \ a_2(T_4(N, 2k + 2)) - a_2(T_4(N, 2k)) \\
    &\geq \frac{\psi(N)^2 16^k (2k+1)}{2} \Bigg[
        \frac{10k+5}{192} - \frac{17}{12} \lrp{\frac{17}{2} \theta_4(N) + \frac14 \theta_1(N)}  - \frac{1}{3} \lrp{\frac{17}{2} \theta_4(N) + \frac14 \theta_1(N)}^2 \\
        &\qquad\qquad\qquad\qquad\qquad\qquad\qquad
        - \frac{73155}{192} \theta_5(N) - \frac{201}{8} \theta_5(N) - \frac{10}{9} \theta_5(N)
        \Bigg] \\
    &=  \frac{\psi(N)^2 16^k (2k+1)}{2} \lrb{ \frac{10k+5}{192} - E(N) },  \label{eqn:temp-T4-a2-diff-EN}
\end{align}
where
\begin{align}
    E(N) &=  \frac{17}{12} \lrp{\frac{17}{2} \theta_4(N) + \frac14 \theta_1(N)}  + \frac{1}{3} \lrp{\frac{17}{2} \theta_4(N) + \frac14 \theta_1(N)}^2 + \frac{234577}{576} \theta_5(N).
\end{align}

Now, from \eqref{eqn:temp-mult-bounds}, we have the explicit numerical bounds
\begin{align}
    \frac{17}{2} \theta_4(N) + \frac14 \theta_1(N) &\leq \frac{17}{2} \cdot 4.862 \cdot N^{-3/4} + \frac{1}{4} \cdot 4.862 \cdot N^{-1/4} \\
    &= 41.327 \cdot N^{-3/4} + 1.2155 \cdot N^{-1/4}, \\
     \qquad \theta_5(N) &\leq N^{-1},
\end{align}
which then yield
\begin{align}
    E(N) &\leq \frac{17}{12} \lrp{41.327 \cdot N^{-3/4} + 1.2155 \cdot N^{-1/4}} \\
    & \ \ \ + \frac13 \lrp{41.327 \cdot N^{-3/4} + 1.2155 \cdot N^{-1/4}}^2 + \frac{234577}{576} N^{-1} \\
    &< \frac{15}{192} \qquad \text{for } N \geq 332,\!427.
\end{align}
This means that by \eqref{eqn:temp-T4-a2-diff-EN}, we have $a_2(T_4(N, 2k + 2)) - a_2(T_4(N, 2k)) > 0$ for all $N \geq 332,\!427$ (independently of $k$). Thus for each odd  $N \geq 332,\!427$, $a_2(T_4(N, 2k))$ is a strictly increasing function of $k$. Then for each of the remaining fixed odd values of $N < 332,\!427$, we will have $\frac{10k+5}{192} > E(N)$ (and hence $a_2(T_4(N, 2k + 2)) - a_2(T_4(N, 2k)) > 0$) for sufficiently large $k$, say $k \geq k_N$. 
This verifies that $a_2(T_2(N, 2k))$ is strictly increasing for $k \geq k_N$.
Then considering $1 \leq k \leq k_N$ such that $s(N,2k) \ge 2$, we verify by computer that each $a_2(T_4(N, 2k)) < a_2(T_4(N, 2k_N))$ and that $a_2(T_4(N, 2k))$ takes no repeated values. See \cite{ross-code} for the code. We note here that in contrast to the behavior of $a_2(T_2(N, 2k))$, the values of $a_2(T_4(N, 2k))$ are not necessarily strictly increasing for small $k$. However, they are still non-repeating.

This shows that for any fixed odd $N$, $a_2(T_4(N,2k))$ takes no repeated values, completing the proof.
\end{proof}

We note that the strategies employed in the proofs of Theorems \ref{theorem:main} and \ref{theorem:main4} apply equally well to any $m \ge 2$.
Based on numerical computation \cite{ross-code} and this general strategy, we conjecture the non-repetition of $a_2(T_m(N,2k))$ for every $m \ge 2$.  
\begin{conjecture} \label{conj:nonrepetition}
Let $m \geq 2$ and $N \ge 1$ coprime to $m$ be fixed, and consider $k \ge 1$ such that $s(N,2k) \ge 2$. Then $a_2(T_m(N,2k))$ takes no repeated values as $k$ varies.
\end{conjecture}

For any given value of $m$, one can verify Conjecture \ref{conj:nonrepetition} using the methods outlined in the proofs of Theorems \ref{theorem:main} and \ref{theorem:main4}. The conjecture will always hold for sufficiently large $k+N$ (see \eqref{eqn:a2-asymptotics} below). Then the finitely many remaining cases can be checked by computer.

We also note a slight difference in these methods depending on whether $m$ is a square or not. The proof of Theorem \ref{theorem:main} showed the monotonic decreasing behavior of $a_2(T_2(N,2k))$, while the proof of Theorem \ref{theorem:main4} showed the (eventual) monotonic increasing behavior of $a_2(T_4(N,2k))$. 
So we can see that $a_2(T_m(N,2k))$ has different asymptotic behavior in these two cases. In particular, by extending the techniques given in the proofs of Theorems \ref{theorem:main} and \ref{theorem:main4}, one can compute (e.g. in \cite[(4.1),(4.2)]{ross-xue}) that for any fixed $m$,
\begin{align} \label{eqn:a2-asymptotics}
    a_2(T_m(N,2k)) = 
    \begin{dcases}
        \frac{\psi(N) m^{2k-2} \sigma_1(m)}{2} \lrb{ 
        -  \frac{2k-1}{12}  
         + O(N^{-1/2+\varepsilon})
        } 
        & \quad \text{if $m$ is a non-square}, \\
        \frac{(2k-1) \psi(N)^2 m^{2k-2}}{2} \lrb{ 
          \frac{2k-1}{144} + O(N^{-1/2+\varepsilon})  
        }
        & \quad \text{if $m$ is a square}.
    \end{dcases}
\end{align}

\section{Non-equality of \texorpdfstring{$a_2(T_2(1,2k))$}{} and \texorpdfstring{$a_2(T_3(1,2k))$}{}}\label{sect:T2T3}

In this section, we prove that for any fixed $k \geq 12$ and $k\ne13$ (that is $s(1,2k)\ge2$), $a_2(T_2(1,2k))\neq a_2(T_3(1,2k))$. Throughout this section, we shall frequently abbreviate $T_m(1, 2k)$ as $T_m$ and $s(1, 2k)$ as $d_{2k}$.


To demonstrate the non-equality of $a_2(T_2)$ and $a_2(T_3)$, we use the following version of the Eichler-Selberg trace formula adapted to level one. The formula here is given by Zagier \cite[Theorem 2]{Eichler-Selber-trace-formula}, although one can check that this follows as a special case of Proposition \ref{eq:eich-selb} for level one:
\begin{equation}
     \Tr T_m=-\frac{1}{2}\sum_{|t|\leq2\sqrt{m}}P_{2k}(t,m)H(4m-t^2)-\frac{1}{2}\sum_{d \mid m}\min(d,m/d)^{2k-1}.\label{LevelOneTrace}
\end{equation}
Here, $H(4m - t^2)$ denotes the Hurwitz class number of discriminant $4m - t^2$. 
Additionally, it is well-known \cite[Theorem 2]{Eichler-Selber-trace-formula} that $P_{2k}(t, m)$ can be written as
\begin{equation} \label{eqn:P2k-formula-rho}
     P_{2k}(t,m)=\frac{\rho^{2k-1}-\overline{\rho}^{2k-1}}{\rho-\overline{\rho}},
\end{equation}
where $\rho+\overline{\rho}=t$ and $\rho\cdot\overline{\rho}=m$. Note also that $P_{2k}(t,m)=P_{2k}(-t,m)$.

For convenience, we introduce here the formula for the dimension $d_{2k}$ of $S_{2k}(\SL_2(\ZZ))$ \cite[p.~88]{diamond2005first}, which is a special case of Proposition \ref{prop:dim-formula} for level one:
\begin{align}\label{eq:dimLevelone}
    d_{2k}=\begin{cases}
        \lfloor \frac{k}{6}\rfloor -1 &k\equiv1\pmod6,\ k > 1,\\ \lfloor \frac{k}{6}\rfloor &{\rm otherwise}.
    \end{cases}
\end{align}
To account for the rounding in this formula, we write
\begin{align}
    d_{2k} = \frac{k}{6} + \delta_k, \label{eq:dim1-correction}
\end{align}
where $-2 \le \delta_k \le 0$.

Using Proposition \ref{prop:a2-formula} to expand $a_2(T_2)$ and $a_2(T_3)$, we have
\begin{align}
    a_2(T_2)=&\frac{1}{2}\Big[(\Tr T_2)^2-\Tr T_4 -  2^{2k-1}d_{2k} \Big], \\
    a_2(T_3)=&\frac{1}{2}\Big[(\Tr T_3)^2-\Tr T_9- 3^{2k-1}d_{2k} \Big].
\end{align}
Then we can then use \eqref{LevelOneTrace} (and the Hurwitz class number values from \cite[A259825]{oeis}) to compute each of the above traces.
We show the work for $\Tr T_3$ in full detail here:
\begin{align}
    \Tr T_3 =& -\frac{1}{2}\sum_{t\leq2\sqrt{3}}P_{2k}(t,3)H(12-t^2)-\frac{1}{2}\sum_{d|3}\min(d,3/d)^{2k-1}\\
    =&-\frac{1}{2}\lrb{P_{2k}(0,3)H(12)+2P_{2k}(1,3)H(11)+2P_{2k}(2,3)H(8)+2P_{2k}(3,3)H(3)}-1\\
    =&-\frac{2}{3}P_{2k}(0,3)-P_{2k}(1,3)-P_{2k}(2,3)-\frac{1}{3}P_{2k}(3,3)-1,\label{TrT3}  \\
    \Tr T_2 =& -\frac{1}{2}P_{2k}(0,2)-P_{2k}(1,2)-\frac{1}{2}P_{2k}(2,2)-1.\label{eq:T2trace}
\end{align}
Next, we have
\begin{align}
    \Tr T_9 =&-\frac{1}{2}\sum_{t\leq6}P_{2k}(t,9)H(36-t^2)-\frac{1}{2}\sum_{d|9}\min(d,9/d)^{2k-1}.\label{TrT9}
\end{align}
The largest $P_{2k}(t, 9)$ terms in the expansion of $\Tr T_9$ \eqref{TrT9} are $P_{2k}(\pm6, 9)$. Using \cite[Lemma 2.3]{previousPaper}, which states that for $t\neq 0$, $P_{2k}(\pm 2t, t^2) = t^{2k-2}(2k - 1)$, we compute $P_{2k}(\pm 6, 9)$ and the Hurwitz class numbers in the above expression, yielding
\begin{align}
    \Tr T_9 =&-\frac{1}{2}\Bigg[ P_{2k}(0,9)H(36)+2P_{2k}(1,9)H(35)+2P_{2k}(2,9)H(32) + 2P_{2k}(3,9)H(27) \\
    &+2P_{2k}(4,9)H(20)+2P_{2k}(5,9)H(11)\Bigg] -(2k-1)\cdot 9^{k-1}H(0) -1-\frac{3^{2k-1}}{2}\\
    =&\frac{2k-1}{12}\cdot 9^{k-1}-\frac{5}{4}P_{2k}(0,9)-2P_{2k}(1,9)-3P_{2k}(2,9)-\frac{4}{3}P_{2k}(3,9)\\
    &-2P_{2k}(4,9)-P_{2k}(5,9)-1-\frac{3^{2k-1}}{2}.
\end{align}
Similarly, we have
\begin{equation}
    \Tr T_4 = \frac{(2k-1)}{12}\cdot 4^{k-1} -\frac{3}{4}P_{2k}(0, 4) -2P_{2k}(1, 4) - \frac{4}{3}P_{2k}(2, 4) - P_{2k}(3, 4) - 2^{2k-1} - 1.\label{eq:TrT4}
\end{equation}
Now, in order to bound each of these traces, we observe the following inequality, which follows immediately from \eqref{eqn:P2k-formula-rho}. For $t^2\neq 4m$,
    \begin{equation}
        |P_{2k}(t, m)| \leq \frac{2m^{k - 1/2}}{\sqrt{|t^2 - 4m|}}\label{eq:P2kbound}.
    \end{equation}
We use this inequality to bound each of the above traces. These next three lemmas all follow a similar procedure.
\begin{lemma}\label{lemma:T3bound}
    We have
    \begin{equation}
        (\Tr T_3)^2 < 22 \cdot 9^{k-1}.
    \end{equation}
\end{lemma}
\begin{proof}
    We use \eqref{eq:P2kbound} to find a bound on $\Tr T_3$. 
The relevant terms are bounded as follows:
\begin{align}
    \frac{2}{3}\left|P_{2k}(0,3)\right|&\leq \frac{2}{3}\cdot\frac{2 \cdot 3^{k-1/2}}{\sqrt{12}}=2\cdot3^{k-2},\\
    \left|P_{2k}(1,3)\right|&\leq \frac{2\cdot3^{k-1/2}}{\sqrt{11}},\\
    \left|P_{2k}(2,3)\right|&\leq\frac{2\cdot3^{k-1/2}}{\sqrt{8}}=\frac{3^{k-1/2}}{\sqrt{2}},\\
    \frac{1}{3}\left|P_{2k}(3,3)\right|&\leq \frac{1}{3}\cdot\frac{2\cdot3^{k-1/2}}{\sqrt{3}}=2 \cdot 3^{k-2}.
\end{align}
These bounds along with $\eqref{TrT3}$ give
\begin{align}
    |\Tr T_3|&=\left|-\frac{2}{3}P_{2k}(0,3)-P_{2k}(1,3)-P_{2k}(2,3)-\frac{1}{3}P_{2k}(3,3)-1\right|\\
    &\leq2\cdot3^{k-2}+\frac{2\cdot3^{k-1/2}}{\sqrt{11}}+\frac{3^{k-1/2}}{\sqrt{2}}+2\cdot3^{k-2}+1\\
    &=3^{k-1}\left(\frac{4}{3}+\frac{2\sqrt{3}}{\sqrt{11}}+\frac{\sqrt{3}}{\sqrt{2}}+\frac{1}{3^{k-1}}\right),
\end{align}
which means that
\begin{align}
    (\Tr T_3)^2\leq&\left(3^{k-1}\left(\frac{4}{3}+\frac{2\sqrt{3}}{\sqrt{11}}+\frac{\sqrt{3}}{\sqrt{2}}+\frac{1}{3^{k-1}}\right)\right)^2\\
    <& \,22 \cdot 9^{k-1},
\end{align}
as desired.
\end{proof}
\begin{lemma}\label{lemma:TrT9}
    Define $C_k$ such that,
    \begin{equation}
        \Tr T_9 = \frac{2k-1}{12}\cdot 9^{k-1} +C_k.
    \end{equation}
    Then $|C_k|< 15\cdot 9^{k-1}$.
\end{lemma}
\begin{proof}
    Recall from \eqref{TrT9} that
    \begin{align}
        \Tr T_9=&\frac{2k-1}{12}\cdot 9^{k-1}-\frac{5}{4}P_{2k}(0,9)-2P_{2k}(1,9)-3P_{2k}(2,9)-\frac{4}{3}P_{2k}(3,9)\\
    &-2P_{2k}(4,9)-P_{2k}(5,9)-1-\frac{3^{2k-1}}{2}.
    \end{align}
    We then use the bound given in \eqref{eq:P2kbound} on each term $P_{2k}(t, 9)$ for which $t \neq \pm 6$:
    \begin{align}
        &\frac{5}{4}|P_{2k}(0,9)|\leq\frac{5\cdot9^{k-1}} {4}, \quad 2|P_{2k}(1,9)|\leq \frac{4\cdot9^{k-1/2}}{\sqrt{35}}, \quad
        3|P_{2k}(2,9)|\leq \frac{9^{k}}{2\sqrt{2}},\\
        &\frac{4}{3}|P_{2k}(3,9)|\leq 8\cdot9^{k-7/4}, \quad
        2|P_{2k}(4,9)|\leq \frac{2\cdot9^{k-1/2}} {\sqrt{5}}, \quad |P_{2k}(5,9)|\leq \frac{2\cdot9^{k-1/2}}{\sqrt{11}}.
    \end{align}
    This yields
    \begin{align}
        |C_k|=&\left|-\frac{5}{4}P_{2k}(0,9)-2P_{2k}(1,9)-3P_{2k}(2,9)-\frac{4}{3}P_{2k}(3,9)-2P_{2k}(4,9)-P_{2k}(5,9)-1-\frac{3^{2k-1}}{2}\right|\\
        \leq&\frac{5\cdot9^{k-1}}{4}+\frac{4\cdot9^{k-1/2}}{\sqrt{35}}+\frac{9^{k}}{2\sqrt{2}}+ 8\cdot9^{k-7/4}+\frac{2\cdot9^{k-1/2}}{\sqrt{5}}+\frac{2\cdot9^{k-1/2}}{\sqrt{11}}+1+\frac{3^{2k-1}}{2}\\
        =&9^{k-1}\left(\frac{5}{4}+\frac{12}{\sqrt{35}}+\frac{9}{2\sqrt{2}}+\frac{8}{9^{3/4}}+\frac{6}{\sqrt{5}}+\frac{6}{\sqrt{11}}+\frac{1}{9^{k-1}}+\frac{3}{2}\right)\\
        <&15\cdot 9^{k-1},
    \end{align}
    as desired.
\end{proof}
\begin{lemma}\label{lemma:TrT4}
    Define $D_k$ such that
    \begin{equation}
        \Tr T_4 = \frac{2k-1}{12}\cdot 4^{k-1} +D_k.
    \end{equation}
    Then $|D_k|< 9\cdot 4^{k-1}$.
\end{lemma}
\begin{proof}
    Similarly to Lemma \ref{lemma:TrT9}, recall first from \eqref{eq:TrT4} that
    \begin{equation}
        \Tr T_4 = \frac{(2k-1)}{12}\cdot 4^{k-1} -\frac{3}{4}P_{2k}(0, 4) -2P_{2k}(1, 4) - \frac{4}{3}P_{2k}(2, 4) - P_{2k}(3, 4) - 2^{2k-1} - 1.
    \end{equation}
    We then bound these $P_{2k}(t, 4)$ terms using \eqref{eq:P2kbound}:
    \begin{align}
        &\frac{3}{4}|P_{2k}(0, 4)|\leq 3\cdot 4^{k-2}, \qquad2|P_{2k}(1,4)|\leq \frac{4^{k+1/2}}{\sqrt{15}},\\
        &\frac{4}{3}|P_{2k}(2, 4)|\leq \frac{4^{k+1/2}}{3\sqrt{3}}, \qquad 2|P_{2k}(1,4)|\leq \frac{4^{k}}{\sqrt{7}}.
    \end{align}
    By using these bounds,
    \begin{align}
        |D_k|&=\left|-\frac{3}{4}P_{2k}(0, 4) -2P_{2k}(1, 4) - \frac{4}{3}P_{2k}(2, 4) - P_{2k}(3, 4)- 2^{2k-1} - 1\right|\\
        &\leq 3\cdot 4^{k-2}+\frac{4^{k+1/2}}{\sqrt{15}}+\frac{4^{k+1/2}}{3\sqrt{3}}+\frac{4^{k}}{\sqrt{7}}+4^{k-1/2}+1\\
        &= 4^{k-1} \lrp{ \frac34 +\frac{8}{\sqrt{15}}+\frac{8}{3\sqrt{3}}+\frac{4}{\sqrt{7}}+2+\frac{1}{4^{k-1}} }\\
        &< 9\cdot 4^{k-1},
    \end{align}
    as desired.
\end{proof}


We now have the tools to prove Theorem \ref{theorem:T2T3notequal}.

{
\renewcommand{\thetheorem}{\ref{theorem:T2T3notequal}}
\begin{theorem}
    For any $k\geq 12$ and $k\neq 13$,
    \begin{equation}
        a_2(T_3(1,2k))< a_2(T_2(1,2k)).
    \end{equation}
\end{theorem}
\addtocounter{theorem}{-1}
}

\begin{proof}
    By \eqref{eq:dim1-correction} and Lemmas \ref{lemma:T3bound} and \ref{lemma:TrT9}, we have 
    \begin{align}
        a_2(T_3) &= \frac{1}{2}\left[(\Tr T_3)^2-\Tr T_9 -3^{2k-1}d_{2k}\right] \\
        &\leq \frac{1}{2}\left[
        22 \cdot 9^{k-1}
        - \frac{2k-1}{12} \cdot 9^{k-1} + 15 \cdot 9^{k-1}
        - 3^{2k-1} \cdot \lrp{\frac{k}{6} - 2} \right] \\
        &= \frac{9^{k-1}}{2}\left[
        22
        - \frac{2k-1}{12}  + 15
        -  \frac{k}{2} + 6 \right] \\
        &\leq \frac{9^{k-1}}{2}\left[
        \frac{-2k}{3} + 44 \right].
    \end{align}
    Similarly, by \eqref{eq:dim1-correction} and Lemma \ref{lemma:TrT4}, we have 
    \begin{align}
        -a_2(T_2) &= \frac{1}{2}\left[-(\Tr T_2)^2+\Tr T_4 +2^{2k-1}d_{2k}\right] \\
        &\leq \frac{1}{2}\left[
        \frac{2k-1}{12} \cdot 4^{k-1} + 9 \cdot 4^{k-1}
        + 2^{2k-1} \cdot \frac{k}{6}  \right] \\
        &\leq \frac{1}{2}\left[
        \frac{k}{2} 4^{k-1} + 9 \cdot 4^{k-1}  \right] \\
        &= \frac{9^{k-1}}{2}\left[
        \lrp{\frac{k}{2}+9} \lrp{\frac49}^{k-1}   \right] \\
        &< \frac{9^{k-1}}{2} \cdot 10.
    \end{align}
    Combining these two inequalities, we obtain
    \begin{align}
        a_2(T_3(1,2k)) - a_2(T_2(1,2k)) \leq \lrb{\frac{-2k}{3} + 54} < 0 \qquad \text{for } k \geq 82.
    \end{align}
    Finally, we verify that $a_2(T_3(1,2k))< a_2(T_2(1,2k))$ for $k < 82$ by computer \cite{ross-code}, completing the proof. 
\end{proof}

\section{Non-repetition for fixed weight and varying prime levels} \label{sect:primelevel}

In this section, we show that for any fixed $k\geq58$, $a_2(T_2(p,2k))\neq a_2(T_2(q,2k))$ when $p$ and $q$ are distinct odd prime numbers. To accomplish this we compute refined trace estimates in the case of prime level. The following three lemmas all follow the same pattern. 

\begin{lemma} \label{lem:primedimension}
    Let $p$ be an odd prime or $1$. Then for $k\ge1$
    \begin{align}
        s(p,2k)=\frac{2k-1}{12}\psi(p)+F_{p,k},
    \end{align}
    where $|F_{p,k}|\le\frac{13}{6}$.
\end{lemma}
\begin{proof}
    By Proposition \ref{prop:dim-formula} and \eqref{eqn:dim-formula-terms-bounds}, we have
    \begin{align}
        |F_{p,k}| &= \lrabs{s(p,2k) - \frac{2k-1}{12} \psi(p)} \\
        &= \Bigg| \lrp{-\frac12  \sum_{d\mid p} \varphi(\gcd(d,p/d)) + \delta_{k,1}} \\
        & - \varepsilon(p/9)\, c_3(k) \prod_{p' \mid p} \lrp{1 + \lrp{\frac{-3}{p'}} } 
        - \varepsilon(p/4)\, c_4(k) \prod_{p' \mid p} \lrp{1 + \lrp{\frac{-4}{p'}} } \Bigg| \\
        &\leq \lrabs{\lrp{-1 \text{ or } \frac{-1}{2}} + \delta_{k,1}} + \lrabs{c_3(k) \prod_{p' \mid p} \lrp{1 + \lrp{\frac{-3}{p'}} } } + \lrabs{ c_4(k) \prod_{p' \mid p} \lrp{1 + \lrp{\frac{-4}{p'}} } } \\
        &\leq 1 + \frac23 + \frac24 = \frac{13}{6},
    \end{align}
    as desired.
\end{proof}

\begin{lemma} \label{lem:primeT2}
    Let $p$ be an odd prime or $1$. Then for $k\ge10$, 
    \begin{align}
        (\Tr T_2(p, 2k))^2\leq 20.8 \cdot 4^{k-1}. 
    \end{align}
\end{lemma}
\begin{proof}
  We follow the proof of \cite[Proposition 4.6]{previousPaper}. Since $p$ is prime,   for each $n\ge1$ such that $n^2\mid t^2-8$ or $n^2\mid t^2-16$, we have the following obvious inequalities on the number of solutions to quadratic congruence modulo $p$:
   \begin{align}
       0\le\mu(t,n,2)\le2, \quad 0\le\mu(t,n,4)\le2.
   \end{align}
   Thus, 
   \begin{align}
       |A_{2,2}|&\le |P_{2k}(2,2)|+2|P_{2k}(1,2)|+|P_{2k}(0,2)|\\
       &\le \frac{2\cdot 2^{k-1/2}}{\sqrt{4}}+2\cdot \frac{2\cdot 2^{k-1/2}}{\sqrt{7}}+\frac{2\cdot 2^{k-1/2}}{\sqrt{8}}.
   \end{align}
   Also,  $|A_{3,2}| \le \sigma_0(p)\le 2$ and $A_{4,2}=0$ since $k\ne1$. Thus we obtain 
   \begin{align}
    |\Tr T_2(N, 2k)|\leq 2^{k-1/2}\left(1+\frac{4}{\sqrt{7}}+\frac{1}{\sqrt{2}}\right)+2 <  4.56 \cdot 2^{k-1} \qquad \text{for } k \ge 10,
   \end{align}
   which yields 
   \begin{align}
       (\Tr T_2(N, 2k))^2 \leq 20.8 \cdot 4^{k-1},
   \end{align}
   as desired.
\end{proof}

\begin{lemma} \label{lem:primeT4}
    Let $p$ be an odd prime or $1$. Then for $k\ge10$,
   \begin{align}
          \Tr T_4(p, 2k) = 4^{k-1} \lrp{ \frac{2k-1}{12}\psi(p) +G_{p,k} }
    \end{align} 
    where $|G_{p,k}|< 13.77$.
\end{lemma}
\begin{proof}
 We follow the proof of \cite[Proposition 4.7]{previousPaper}. The main term of  $\Tr T_4(N, 2k) $ is given by
    \begin{align}
        A_{1,4}=\frac{2k-1}{12}\psi(p) 4^{k-1}.
    \end{align}
    Then we also have
    \begin{align}
        |A_{2,4}|&\le 2|P_{2k}(3,4)|+4 |P_{2k}(1,4)|+\left(2+\frac{2}{3}\right)|P_{2k}(2,4)|+\frac{3}{2}|P_{2k}(0,4)|\\
        &\le 2\cdot \frac{2\cdot 4^{k-1/2}}{\sqrt{7}}+4\frac{2\cdot 4^{k-1/2}}{\sqrt{15}}+\frac{8}{3}\cdot \frac{2\cdot 4^{k-1/2}}{\sqrt{12}}+\frac{3}{2}\cdot\frac{2\cdot 4^{k-1/2}}{\sqrt{16}}\\
        &< 11.76 \cdot 4^{k-1}, \\
        |A_{3,4}|&\le 2\sigma_0(p)+2^{2k-1} \\
        &\le4+2^{2k-1} \\
        &< 2.01 \cdot 4^{k-1}  \text{ for } k\ge 10, \\
        A_{4,4} &= 0 ,
    \end{align}
    which immediately yields the desired result.
\end{proof}

We can now use these three lemmas to prove  Theorem \ref{theorem:primelevel}.

{
\renewcommand{\thetheorem}{\ref{theorem:primelevel}}
\begin{theorem} 
 If $k\ge 58$, then
    \begin{equation}
      a_{2}(T_2(p,2k))> a_{2}(T_2(q,2k))
    \end{equation}
 for all odd primes $p< q$, where $p$ can be taken to be $1$.
\end{theorem}
\addtocounter{theorem}{-1}
}

\begin{proof}

 By Proposition \ref{prop:a2-formula} and Lemmas \ref{lem:primedimension}-\ref{lem:primeT4}
    \begin{align}
    a_{2}(T_2(p,2k))
    &=\frac{1}{2} \Big[(\Tr T_2(p, 2k))^2-\Tr T_4(p, 2k) - 2^{2k-1}s(p, 2k) \Big] \\
    &= \frac{1}{2} \lrb{ \Tr T_2(p, 2k))^2 - 
    4^{k-1} \lrp{ \frac{2k-1}{12}\psi(p) + G_{p,k}} -  2^{2k-1} \lrp{\frac{2k-1}{12} \psi(p) + F_{p,k} }  } \\
    &= \frac{4^{k-1}}{2} \lrb{ \frac{\Tr T_2(p, 2k))^2}{4^{k-1}} - 
    \lrp{ \frac{2k-1}{12}\psi(p) + G_{p,k}} -  2 \lrp{\frac{2k-1}{12} \psi(p) + F_{p,k} }  } \\
    &= \frac{4^{k-1}}{2} \lrb{ - 
     \frac{2k-1}{4}\psi(p) - G_{p,k} -  2  F_{p,k} +  \frac{\Tr T_2(p, 2k))^2}{4^{k-1}}   }. 
    \end{align}
    Then using Lemmas \ref{lem:primedimension}-\ref{lem:primeT4} and the bound $\psi(q)-\psi(p)\ge q-p\ge2$, we have
    \begin{align}
         & \ \ \ a_{2}(T_2(p,2k))-a_{2}(T_2(q,2k))\\
         &= \frac{4^{k-1}}{2} \lrb{
            \frac{2k-1}{4}\lrp{\psi(q)-\psi(p)} + G_{q,k}-G_{p,k} +  2 F_{q,k} - 2 F_{p,k} + \frac{\Tr T_2(p, 2k))^2 - \Tr T_2(q, 2k))^2}{4^{k-1}} 
        } \\
        &\geq \frac{4^{k-1}}{2} \lrb{
            \frac{2k-1}{4} \cdot 2 + G_{q,k}-G_{p,k} +  2 F_{q,k} - 2 F_{p,k} -\frac{ \Tr T_2(q, 2k))^2}{4^{k-1}} 
        } \\
        &\geq \frac{4^{k-1}}{2} \lrb{
            \frac{2k-1}{2} - 2\cdot 13.77 - 4 \cdot \frac{13}{6} - 20.8
        },
    \end{align}
    which is greater than $0$ for $k\ge 58$. 
\end{proof}

\section{Distinguishing Hecke eigenforms} \label{sect:distinguish}

Recall that Maeda's conjecture states that the characteristic polynomials of the Hecke operators $T_m(1,2k)$ are irreducible for all $m \ge 2$, $k \ge 1$.
In this section, we use the non-repetition results proven in Theorems \ref{theorem:main} and \ref{theorem:main4} to extend the main result of \cite{Vilardi2018}.
Theorem 1.2 of \cite{Vilardi2018} showed that, assuming Maeda's conjecture, a normalized Hecke eigenform of level one is determined by its $2$nd Fourier coefficient.
A natural question one might then ask is if the same result holds for the other Fourier coefficients. 
In fact, this was recently conjectured by Xue and Zhu in \cite{xue-zhu}.
\begin{conjecture}[{\cite[Conjecture 3.5]{xue-zhu}}] \label{conj:distinguish}
    For any fixed $m \geq 2$, a normalized Hecke eigenform of level one is determined by its $m$-th Fourier coefficient.
\end{conjecture}

The papers \cite{Vilardi2018} and \cite{xue-zhu} used the non-repetition of $\Tr T_m(1, 2k)$ to prove the $m=2$ and $m=3$ cases of Conjecture \ref{conj:distinguish}, respectively (conditional on Maeda's conjecture). 
We present a slightly different approach using the non-repetition of $a_2(T_m(1,2k))$; assuming Maeda's conjecture, we use Theorems \ref{theorem:main} and \ref{theorem:main4} to verify Conjecture \ref{conj:distinguish} for $m=2,4$.

{
\renewcommand{\thetheorem}{\ref{theorem:distinguish-eigenforms}}
\begin{theorem} 
    Let $m = 2 \text{ or }4$ be fixed, and let $f=\sum_{n\ge1} c_f(n) q^n\in S_{2k_1}(\SL_2(\ZZ))$ and $g=\sum_{n\ge1} c_g(n) q^n\in S_{2k_2}(\SL_2(\ZZ))$ be normalized Hecke eigenforms. Assume that the characteristic polynomials of the $T_m(1,2k_i)$ are irreducible. Then
    \begin{equation}
    f=g \quad \text{iff} \quad c_f(m) = c_g(m).
    \end{equation}
\end{theorem}
\addtocounter{theorem}{-1}
}

\begin{proof}
    For weights $2k$ with $s(1,2k) \neq 0$, let $\Egv_{2k}$ denote the set of eigenvalues of $T_m(1,2k)$ (counting multiplicity). We observe that these $\Egv_{2k}$ are all distinct.
    For $\Egv_{2k}$ of size at least $2$, this observation immediately follows from Theorems \ref{theorem:main} and \ref{theorem:main4}.
    The $\Egv_{2k}$ of size $1$ come from $2k = 12,16,18,20,22,26$. And we can verify directly that each of these $\Egv_{2k}$ is distinct \cite{ross-code}.
    
    We now show the desired result. The forward direction is immediate.
    Then suppose that $c_f(m) = c_g(m)$, and we will show that $f=g$. Note that the irreducibility of characteristic polynomials means that the $\Egv_{2k_i}$ will be precisely the Galois orbit of any one eigenvalue. Thus $c_f(m) = c_g(m)$ means that $\Egv_{2k_1} = \Egv_{2k_2}$. Then by the above observation, this means that in particular, we must have $k_1 = k_2$. 
    Thus since $k_1=k_2$, $c_f(m)$ and $c_g(m)$ are two eigenvalues of $T_m(1,2k_1)$ associated to $f$ and $g$. But since the characteristic polynomial of $T_m(1,2k_1)$ has no repeated roots, we must have $f=g$. This completes the proof.
\end{proof}

We note that our approach is easy to generalize to any given $m \geq 2$. For any given $m \ge 2$, our methods prove the non-repetition of $a_2(T_m(1,2k))$ for sufficiently large $k$ (see the discussion below Conjecture \ref{conj:nonrepetition}). Assuming Maeda's conjecture, this verifies Conjecture \ref{conj:distinguish} for eigenforms of sufficiently large weight. Then the finitely many remaining cases can be checked by computer.
This is an advantage over the approaches of \cite{Vilardi2018} and \cite{xue-zhu}, which do not easily generalize to larger values of $m$. In particular, we are not aware of any general strategy to show the non-repetition of $\Tr T_m(1,2k)$.

In fact, this is an instance of a general phenomenon: The second coefficient $a_2(T_m(N,2k))$ appears to be an easier object to study than the trace 
$\Tr T_m(N,2k)$ (in part, due to the fact that $a_2(T_m(N,2k))$ has the asymptotic behavior described in \eqref{eqn:a2-asymptotics}, while $\Tr T_m(N,2k)$ does not seem to follow any sort of asymptotic behavior).
For example, it is known that $a_2(T_m(N,2k))$ is nonvanishing for sufficiently large $k$ \cite[Theorem 1.2]{ross-xue}. However, no such result exists so far for $\Tr T_m(N,2k)$ (even in level one).
For another example, the non-repetition of $a_2(T_2(N,2k))$ shown in Theorem \ref{theorem:main} gives much stronger results than the non-repetition of $\Tr T_2(1,2k)$ shown in \cite[Theorem 1]{chirjorznorepeat}. First, Theorem \ref{theorem:main} proves the monotonicity of $a_2(T_2(N,2k))$ in $k$, while the trace $\Tr T_2(1,2k)$ is by no means monotonic. Second, Theorem \ref{theorem:main} addresses every odd level, instead of confining to level one (and the method of \cite[Theorem 1]{chirjorznorepeat} does not easily extend to higher levels).

We also note that the other coefficients of Hecke polynomials could similarly be used here. Based on the results of this work, one should be able to show the non-repetition of the other even-indexed coefficients $a_{2j}(T_m)$.
In particular, one could apply the methods of this paper to obtain explicit bounds for the asymptotic behavior of $a_{2j}(T_m)$ proven in \cite[Theorems 1.1, 1.2]{ross-xue-rth-coeff}, then use these explicit bounds to prove the non-repetition of $a_{2j}(T_m)$.

\section*{Acknowledgements}
This research was supported by NSA MSP grant H98230-23-1-0020.

\bibliographystyle{plain}
 


\providecommand{\bysame}{\leavevmode\hbox
to3em{\hrulefill}\thinspace}

\bibliography{bibliography.bib}

\end{document}